\documentclass{amsart}

\usepackage{amssymb}
\usepackage{amsmath, longtable}

\newtheorem{theorem}{Theorem}[section]
\newtheorem{lemma}[theorem]{Lemma}
\newtheorem{proposition}[theorem]{Proposition}
\newtheorem{corollary}[theorem]{Corollary}

\newtheorem{_definition}[theorem]{Definition}
\newenvironment{definition}{\begin{_definition}\rm}{\end{_definition}}

\newtheorem{_remark}[theorem]{\it Remark}

\newtheorem{_claim}[theorem]{Claim}

\numberwithin{equation}{section}
\numberwithin{table}{section}
\numberwithin{figure}{section}

\renewcommand{\P}{\mathord{\mathbb  P}}
\newcommand{\Z}{\mathord{\mathbb Z}}

\newcommand{\Q}{\mathord{\mathbb Q}}

\newcommand{\TTT}{\mathord{\mathcal T}}

\newcommand{\set}[2]{\{\; {#1} \; \mid \; {#2} \;  \}}

\newcommand{\wt}{\widetilde}

\newcommand{\sprime}{\sp\prime}
\newcommand{\sptimes}{\sp\times}

\newcommand{\dual}{\sp{\vee}}

\newcommand{\sperp}{\sp\perp}

\newcommand{\OL}{\mathord{\mathcal{L}}}
\newcommand{\OLL}{\mathord{\tilde{\mathcal{L}}}}

\newcommand{\rmand}{\textrm{and}}

\newcommand{\quand}{\quad\rmand\quad}

\newcommand{\NS}{\mathord{\rm{NS}}}
\newcommand{\ch}{\operatorname{\mathord{\rm{char}}}\nolimits}
\newcommand{\disc}{\operatorname{\mathord{\rm{disc}}}\nolimits}
\newcommand{\Roots}{\mathord{\rm{roots}}}

\newcommand{\Hom}{\mathord{\rm{Hom}}}
\newcommand{\Aut}{\mathord{\rm{Aut}}}
\newcommand{\Sing}{\mathord{\rm{Sing}}}
\newcommand{\rank}{\operatorname{\mathord{\rm{rank}}}\nolimits}

\newcommand{\lat}{\Lambda}

\newcommand{\roots}{\Roots}
\newcommand{\Ex}{{\text{\rm Ex}}}

\begin{document}
\title[Supersingular $K3$ surfaces]{Dynkin diagrams of rank $20$ on supersingular $K3$ surfaces}

\author{Ichiro Shimada}
\address{
Department of Mathematics,
Graduate School of Science,
Hiroshima University,
1-3-1 Kagamiyama,
Higashi-Hiroshima,
739-8526 JAPAN
}
\email{shimada@math.sci.hiroshima-u.ac.jp
}

\author{De-Qi Zhang}
\address{
Department of Mathematics,
National University of Singapore,
10 Lower Kent Ridge Road,
Singapore 119076, SINGAPORE
}
\email{matzdq@nus.edu.sg
}
\begin{abstract}
We classify normal supersingular $K3$ surfaces $Y$ with total Milnor number $20$
in characteristic $p$,
where $p$ is an odd prime that does not divide the discriminant of  the
Dynkin type of the rational double points on $Y$.
This paper appeared in preprint form in the home page of the first author in the year 2005.
\end{abstract}


\maketitle

\section{Introduction}
A \emph{Dynkin type} is, by definition,
a finite formal sum of the symbols $A_l (l\ge 1)$, $D_m (m\ge 4)$ and
$E_n (n=6,7,8)$ with non-negative integer coefficients.
For a Dynkin type $R$,
we denote by $L(R)$ the \emph{negative-definite} lattice
whose intersection matrix is $(-1)$ times the Cartan matrix of type $R$.
We denote by $\rank (R)$ the rank of $L(R)$, and by $\disc (R)$ the discriminant of $L(R)$.
\par
A \emph{normal $K3$ surface} is a normal surface whose minimal resolution is a $K3$ surface.
It is well-known that
a normal  $K3$ surface has only rational double points as its singularities
(\cite{Artin1962, Artin1966}).
Hence we can associate  a Dynkin type to the singular locus $\Sing (Y)$ of a normal $K3$ surface $Y$.
Recall that
the Milnor number of a rational double point of type $A_n$ (resp.~$D_n$,~resp.~$E_n$)
is $n$.
Hence
the rank of the Dynkin type of $\Sing (Y)$ is equal to
the sum of Milnor numbers of singular points on $Y$, that is, the \emph{total Milnor number of $Y$}.
In particular,
it is at most $21$.
\par
%
%
If the total Milnor number of a normal $K3$ surface $Y$ is $\ge 20$,
then the minimal resolution $X$ of $Y$ has Picard number $\ge 21$, and hence is a supersingular $K3$ surface
(in the sense of Shioda~\cite{Shioda}).
In~\cite[Theorem~3.7]{Goto}, Goto proved that a normal $K3$ surface $Y$ with total Milnor number
$21$ exists only when  the characteristic of the base field divides the discriminant of the Dynkin type of $\Sing (Y)$.
In~\cite{Shimada},
the first author made the complete list of the pairs $(R, p)$ of a Dynkin type $R$  of rank $21$ and a prime integer  $p$
such that
$R$ is  the  Dynkin type of the singular locus  of a normal $K3$ surface
in characteristic $p$.
\par
%
%
In this paper,
we investigate normal $K3$ surfaces with total Milnor number $20$.
\begin{definition}
Let $R$ be a Dynkin type of rank $20$.
A prime integer $p$ is called
an \emph{$R$-supersingular $K3$ prime}
if it satisfies the following:
\begin{itemize}
\item[(i)] $p$ is odd and does not divide $\disc (R)$, and
\item[(ii)] there exists a normal $K3$ surface $Y$ defined over an algebraically closed field
of characteristic $p$ such that $\Sing (Y)$ is of type $R$.
\end{itemize}
\end{definition}
The \emph{Artin invariant} of a supersingular $K3$ surface $X$ in characteristic $p$ is
the positive integer $\sigma$ such that the discriminant of the N\'eron-Severi lattice $\NS(X)$
of $X$ is equal to
$-p^{2\sigma}$. (See~\cite{Ar}).
We will prove that,
if $p$ is an $R$-supersingular $K3$ prime
for a Dynkin type $R$ with $\rank (R)=20$,
and if $Y$ is a normal $K3$ surface in the condition (ii) above,
then the Artin invariant of the minimal resolution of $Y$ is $1$.
It is known that,
for each $p$, the supersingular $K3$ surface with Artin invariant $1$ is unique up to isomorphisms
(\cite{Ogus, DKondo}).
Therefore the condition (i) and (ii) above is equivalent to (i) and the following:
\begin{itemize}
\item[(ii)${}\sprime$]
the supersingular $K3$ surface $X_p$ in characteristic $p$ with Artin invariant $1$
is birational to a normal $K3$ surface $Y$ such that  $\Sing (Y)$ is of type $R$.
\end{itemize}
In this paper, we present an algorithm to determine
the set of $R$-supersingular $K3$ primes for a given Dynkin type $R$ of rank $20$.
As a corollary, we prove the following.
\begin{theorem}\label{thm:Sigma1}
Let $R$ be a Dynkin type of rank $20$,
and let $a_R$ be the product of the odd prime divisors of $\disc (R)$.
We put $b_R:=8 a_R$ if $\disc (R)$ is even,
while $b_R:=a_R$ if $\disc (R)$ is odd.
Then there exists a subset $\Sigma_R$ of $(\Z/b_R\Z)\sptimes$
such that
a  prime integer $p$ is an $R$-supersingular $K3$ prime
if and only if
$p \bmod b_R\in \Sigma_R$.
\end{theorem}
In fact,
we have a result finer than above.
Let $Y$ be a normal supersingular $K3$ surface in characteristic $p\ne 2$
such that $\Sing (Y)$ is of Dynkin type $R$ with $\rank (R)=20$,
and let $X\to Y$ be the minimal resolution of $Y$.
We denote by $L_Y$ the sublattice of the N\'eron-Severi lattice $\NS (X)$ of $X$
generated by the classes of the exceptional curves of $X\to Y$.
Then $L_Y$ is isomorphic to $L(R)$.
Let $T_Y$ denote the orthogonal complement of $L_Y$ in $\NS (X)$.
Then $T_Y$ is an even indefinite lattice of rank $2$.
Our key observation is the following:
\begin{equation}\label{eq:observation}
tt\sprime\in p\Z\quad\textrm {for all}\quad t, t\sprime\in T_Y,
\end{equation}
where $tt\sprime\in \Z$ is the intersection number of the classes $t$ and $t\sprime$ in $\NS (X)$.
Thus we can define an indefinite lattice $T\sprime_Y$ of rank $2$ by
introducing a new bilinear form
$$
(t, t\sprime)_{T\sprime_Y}:=\frac{1}{p}(t t\sprime).
$$
on the $\Z$-module underlying $T_Y$.
It turns out that $\disc (T\sprime_Y)$ divides $\disc (R)$.
Note that, since $p$ is odd, $T\sprime_Y$ is an even lattice.
Let $\wt{L}_Y$ be the orthogonal complement of $T_Y$ in $\NS (X)$.
Then $\wt{L}_Y$ is an even overlattice of $L_Y$
such that the set $\Roots (\wt{L}_Y)$ of roots in $\wt{L}_Y$
coincides with the set $\Roots (L_Y)$ of roots in $L_Y$.
The following is a refinement of Theorem~\ref{thm:Sigma1}.
\begin{theorem}\label{thm:Sigma2}
Let $R$ be a Dynkin type of rank $20$,
let $T\sprime$ be an even indefinite lattice of rank $2$
such that $\disc (T\sprime)$ divides $\disc (R)$,
and let $\wt{L}$ be an even overlattice of $L(R)$ such that $\Roots (\wt{L})=\Roots (L(R))$.
Then there exist a subset $S_l$ of $\{1, -1\}$ for each odd prime divisor $l$ of $\disc (R)$,
and a subset $S_2$ of $\{1,3,5,7\}$,
such that  the following holds.
Let $p$ be an odd prime that does not divide $\disc(R)$.
Then there exists a normal $K3$ surface $Y$ in characteristic $p$
with $\Sing (Y)$ being of type $R$
such that $T\sprime_Y\cong T\sprime$ and $\wt{L}_Y\cong \wt{L}$
if and only if
\begin{equation}\label{eq:pcond}
\vtop{
\hbox{$\left(\frac{p}{l}\right)\in S_l$ for each odd prime divisor $l$ of $\disc(R)$, and}
\hbox{$p \bmod 8 \in S_2$.}
}
\end{equation}
If $\disc (R)$ is odd, then we have $S_2=\{1,3,5,7\}$.
\end{theorem}
Using computational algebra system {\tt Maple}, we have made
the complete list of $R$-supersingular $K3$ primes,
which  is
too large to be included in this paper.
It is available from the first author's home page
\begin{center}
\verb+http://www.math.sci.hiroshima-u.ac.jp/~shimada/K3.html+
\end{center}
in the plain text format.
From this list, we derive  the following fact:
\begin{theorem}\label{thm:half}
For each Dynkin type $R$ with $\rank (R)=20$,
the set of $R$-supersingular $K3$ primes is either empty or
has  a natural density $1/2$.
\end{theorem}
It would be very nice if Theorem~\ref{thm:half} is proved,
not by brute calculations of making the complete list, but by some
geometric reasonings.
\par
As another corollary of the key observation~\eqref{eq:observation},
we obtain the following:
\begin{corollary}\label{cor:degree}
Let $Y\subset \P^N$ be a normal supersingular $K3$ surface of total Milnor number $20$
such that $\Sing (Y)$ is of type $R$.
If the characteristic $p$ of the base field is odd and does not divide $\disc (R)$,
then the degree of $Y$ is divisible by $2p$.
\end{corollary}
Indeed the class of the pull-back of the hyperplane section of $Y$ to $X$ is
contained in $T_Y$.
Note that,
if $R$ is of rank $20$, then every prime divisor of $\disc(R)$ is $\le 19$.
Combining Corollary~\ref{cor:degree}
with~\cite[Theorem~3.7]{Goto},
we obtain the following:
\begin{corollary}\label{cor:mu19}
Let $Y$ be a normal $K3$ surface of degree $d$
in characteristic $p>19$.
If $p$ does not divide $d$,
then
the total Milnor number  of $Y$ is $\le 19$.
\end{corollary}
In particular, a sextic plane curve $C\subset \P^2$ or a quartic surface
$S\subset \P^3$  in characteristic $p>19$ with only rational double points as its singularities
has total Milnor number $\le 19$.
Yang~\cite{Yang1, Yang2} classified all possible configurations of rational double points
on sextic plane curves and quartic surfaces in characteristic $0$.
It would be interesting to investigate
Yang's  classification   in characteristic $p>19$.
See~\cite{ShimadaNormal} for a result on this problem.
\par
%
%
%
%
In our previous paper~\cite{ShimadaZhang},
we have proved that normal $K3$ surfaces with ten ordinary cusps
exist only in characteristic $3$.
This implies that
the set of $10A_2$-supersingular $K3$ primes is empty.
More generally, the proof of Dolgachev-Keum \cite[Lemma 3.2]{DK}
shows that, if $\disc (R)$ is a perfect square integer, then
there exist no $R$-supersingular $K3$ primes.
(See~Lemma~\ref{sqfree}.)
\par
There are $3058$ Dynkin types  of rank $20$.
Among them, there exist
$2437$ Dynkin types $R$ such that $\disc (R)$ is not a perfect square integer,
and $483$ Dynkin types with non-empty set of $R$-supersingular $K3$ primes.
\par
%
%
This paper is organized as follows.
In~Section~\ref{sec:lattice},
we reduce the problem of determining $R$-supersingular $K3$ primes
to the calculation of overlattices of $L(R)$ and their quadratic forms.
In~Section~\ref{sec:fqf},
we investigate how the multiplications by odd prime integers affects
the isomorphism classes of finite quadratic forms.
In~Section~\ref{sec:algorithm},
we  present an algorithm to calculate the set of $R$-supersingular $K3$ primes.
In the last section,
we explain the algorithm in detail by using an example.
\par
%
The study of the cases where $p$ is $2$ or divides $\disc (R)$
seems to need more subtle methods, and hence we do not treat these cases.
\par \vskip 1pc \noindent
{\bf Acknowledgement.}
We would like to thank the referees for the valuable suggestions and pointing out the reference
of C. T. C. Wall \cite{MR0156890} about quadratic forms.
\section{The N\'eron-Severi lattices of supersingular $K3$ surfaces}\label{sec:lattice}
A free $\Z$-module $\lat$ of finite rank with a non-degenerate
symmetric bilinear form
$\lat\times \lat \to \Z$
is called a \emph{lattice}.
Let $\lat$ be a lattice.
The \emph{dual lattice} $\lat\dual$ of $\lat$ is the $\Z$-module
$\Hom (\lat, \Z)$.
Then $\lat$ is naturally embedded into $\lat\dual$ as a submodule of finite index.
There exists a natural  $\Q$-valued symmetric bilinear form
on $\lat\dual$
that extends the $\Z$-valued symmetric bilinear form on $\lat$.
An \emph{overlattice} of $\lat$ is a submodule $N$ of $\lat\dual$
containing $\lat$
such that the bilinear form on $\lat\dual$
takes values in $\Z$ on $N\times N$.
If $\lat$ is a sublattice of a lattice $\lat\sprime$ with finite index,
then $\lat\sprime$ is embedded into $\lat\dual$ in a natural way,
and hence $\lat\sprime$ can be regarded as an overlattice of $\lat$.
\par
We say that $\lat$ is \emph{even} if $u^2\in 2\Z$ holds for every $u\in \lat$.
Let $\lat$ be an even negative-definite lattice.
A vector $r\in \lat$ is called a \emph{root}
if $r^2=-2$.
We denote by $\Roots (\lat)$ the set of roots in $\lat$.
Let $R$ be a Dynkin type.
Recall that $L(R)$ is the  negative definite root lattice of type $R$.
We put
\begin{eqnarray*}
\OLL(R)&:=&\set{\wt{L}}{\textrm{$\wt{L}$ is an even overlattice of $L(R)$}},\\
\OL(R)&:=&\set{\wt{L}\in \OLL(R)}{\Roots (\wt{L})=\Roots(L(R))}.
\end{eqnarray*}
Remark that we consider $\OLL(R)$ as a subset of
the set of submodules of $L(R)\dual$, and \emph{not} up to isometries of lattices.
\par
Let $D$ be a finite abelian group.
A \emph{quadratic form} $q$ on $D$ is, by definition,
a map $q: D\to \Q/2\Z$ that satisfies the following:
\begin{itemize}
\item[(i)] $q(nx)=n^2 q(x)$ for any $x\in D$ and any $n\in \Z$, and
\item[(ii)] the map $b[q]: D\times D\to \Q/\Z$ defined by
$$
b[q](x, y):= (q(x+y)-q(x)-q(y))/2
$$
is  a symmetric bilinear
form on $D$.
\end{itemize}
See Wall~\cite{MR0156890} for the complete classification of quadratic forms on finite abelian groups.
Let $q$ be a quadratic form on a finite abelian group $D$,
and let $H$ be a subgroup of $D$.
We put
$$
H\sperp :=\set{x\in D}{b[q](x, y)=0\;\;\textrm{for any}\; y\in H}.
$$
We say that $q$ is \emph{non-degenerate} if $D\sperp=\{0\}$.
\par
Let $\lat$ be an even lattice.
The finite abelian group $\lat\dual/\lat$ is called the \emph{discriminant group}
of $\lat$, and is denoted by $D_\lat$.
We define a quadratic form $q_\lat$ on $D_\lat$ by
$$
q_\lat (\bar x)=x^2\;\bmod\,2\Z\qquad (\textrm{where}\;\bar x:= x\bmod\,\lat\in D_\lat\;\textrm{for $x\in
\lat\dual$}),
$$
and call $q_\lat$ the \emph{discriminant form} of $\lat$.
It is easy to see that $q_\lat$ is non-degenerate, and that
$$
|D_\lat|=|\disc (\lat)|
$$
holds.
By Nikulin~\cite{Nikulin},
the map
$$
N\mapsto N/\lat
$$
induces a
 bijection  from the set of even overlattices of $\lat$ to
the set of isotropic subgroups of $(D_\lat, q_{\lat})$.
In particular,
we can calculate the set of even overlattices of
a given lattice by calculating the  set of isotropic subgroups
of its discriminant forms.
Let $N$ be an even overlattice of $\lat$.
Then we have a natural sequence of inclusions
$$
\lat\;\;\hookrightarrow\;\;N\;\;\hookrightarrow\;\; N\dual\;\; \hookrightarrow\;\; \lat\dual.
$$
Therefore $\disc (N)$ divides $\disc(\lat)$,
and the exponent of $D_N$ divides the exponent of $D_\lat$.
For a prime $p$,
we denote by $(D_\lat)_p$ and $(D_\lat)_{p\sprime}$ the $p$-part and the prime-to-$p$ part of
$D_\lat$, and by $(q_\lat)_p$ and $(q_\lat)_{p\sprime}$ the restrictions of $q$ to
$(D_\lat)_p$ and to $(D_\lat)_{p\sprime}$, respectively.
We have the following
orthogonal decomposition:
$$
(D_\lat, q_\lat)=((D_\lat)_p, (q_\lat)_p)\oplus ((D_\lat)_{p\sprime}, (q_\lat)_{p\sprime}).
$$
%
%
%
\par
We now state the main theorem of this section.
\begin{theorem}\label{Th2.1} 
Let $R$ be a Dynkin type with $\rank (R)=20$,
and let $p$ be an odd prime  that does not divide $\disc (R)$.
Then the following three conditions are equivalent:
\begin{itemize}
\item[(1)] $p$ is an $R$-supersingular $K3$ prime.
\item[(2)]
The unique supersingular $K3$ surface $X_p$ of Artin invariant
$1$
in characteristic $p$
is birational to a normal $K3$ surface $Y$ such that $\Sing(Y)$ is of Dynkin type $R$.
\item[(3)] There exist
an overlattice $\wt{L}\in \OL (R)$ and  a lattice $T\sprime$ of rank $2$ with signature $(1,1)$
such that $(D_{T\sprime}, p q_{T\sprime})$ is isomorphic to $(D_{\wt{L}}, -q_{\wt{L}})$,
where $p q_{T\sprime}$ is the discriminant form of $T\sprime$ multiplied by $p$.
\end{itemize}
\end{theorem}
%
%
%
%
%
%
A lattice $\Lambda$ is said to be \emph{$p$-elementary}
if its discriminant group $D_{\Lambda}$ is a $p$-elementary  group.
The following results due to
Artin~\cite{Ar} and Rudakov-Shafarevich~\cite{RS2}
reduce our geometric problem to calculations of lattices and finite quadratic forms.
\begin{theorem}[Artin~\cite{Ar} and Rudakov-Shafarevich~\cite{RS2}]
The N\'eron-Severi lattice of a supersingular $K3$ surface in characteristic $p$ is $p$-elementary.
\end{theorem}
Combining this result with the classification of indefinite lattices,
we have the following.
\begin{corollary}[Rudakov-Shafarevich~\cite{RS2}]
The isomorphism class of the N\'eron-Severi lattice $\NS(X)$ of a supersingular $K3$ surface $X$  is
uniquely
determined by the characteristic $p$ of the base field and the Artin invariant of $X$.
\end{corollary}
\begin{corollary}[Rudakov-Shafarevich~\cite{RS2}]\label{cor:RS2}
Suppose that $p$ is odd.
If $\Lambda$ is an even $p$-elementary
lattice of signature $(1, 21)$ and discriminant $-p^{2\sigma}$,
then $\Lambda$ is isomorphic to the N\'eron-Severi lattice of a supersingular $K3$ surface
in characteristic $p$ with Artin invariant $\sigma$.
\end{corollary}
%
%
%
The following easy result will be used in the proof of Theorem \ref{Th2.1}.
\begin{lemma}\label{TT'}
Let $T' = \Z t_1'\oplus \Z t_2'$ be a  lattice of rank $2$ with the intersection matrix
$(t_i'. t_j') = (t_{ij}')$.
Let $p$ be a prime that does not divide $\disc(T')$.
Define a lattice $T = \Z t_1\oplus  \Z t_2$ so that the intersection matrix
$(t_i . t_j) = (t_{ij})$ with $t_{ij} = p t_{ij}'$.
Then the following hold.
\begin{itemize}
\item
[(1)] $((D_T)_{p'}, (q_{T})_{p'}) \cong (D_{T\sprime}, p q_{T\sprime})$.
\item
[(2)] There exist positive integers $\ell_1$ and $\ell_2$ such that
$$T^{\vee}/T \cong \Z/(p \ell_1) \oplus \Z/(p \ell_2),
\hskip 0.5pc
(T')^{\vee}/T' \cong \Z/(\ell_1) \oplus \Z/(\ell_2).$$
\end{itemize}
\end{lemma}
\begin{proof}
Since $\rank(T') = 2$, we can write
$(T')^{\vee}/T' \cong \Z/(\ell_1) \oplus \Z/(\ell_2)$ so that $\ell_i > 0$, \
$\ell_1 | \ell_2$ and $\disc(T') = \det(t_{ij}') = \ell$
with $|\ell| = \ell_1 \ell_2$.
We can calculate the dual bases of $T'$ and $T$ as follows,
where $s_{11} = t_{22}'$, $s_{22} = t_{11}'$ and $s_{12} = s_{21} = -t_{12}' = -t_{21}'$:
$$((t_1')^*, (t_2')^*) = (t_1', t_2') (t_{ij}')^{-1} = \frac{1}{\ell} (t_1', t_2') (s_{ij}),$$
$$(t_1^*, t_2^*) = (t_1, t_2) (t_{ij})^{-1} = \frac{1}{p^2 \ell} (t_1, t_2) ( p s_{ij}).$$
Note that $(T^{\vee}/T)_{p'}$ is generated by
(cosets of) the two coordinates of the vector:
$$(p t_1^*, p t_2^*) = \frac{p}{\ell} (t_1, t_2) (s_{ij}).$$
Set $b_{T'} = b[q_{T'}]$, etc. Then
$$(b_{T'}((t_i')^* , (t_j')^*)) = (t_{ij}')^{-1} = \frac{1}{\ell}(s_{ij}),$$
$$(b_T(t_i^* , t_j^*)) = (t_{ij})^{-1} = \frac{1}{p^2 \ell}(p s_{ij}).$$
One can check that the following is an isomorphism of abelian groups:
$$(T')^{\vee}/T'  \rightarrow (T^{\vee}/T)_{p'}$$
$$ (t_i')^* + T' \mapsto p t_i^* + T.$$
Under the identification via this map, we have $p q_{T\sprime} = (q_T)_{p'}$.
This proves (1). Clearly, $\disc(T) = \det(t_{ij}) = p^2 \disc(T')$.
Also the expression of the dual basis shows that
$(T^{\vee}/T)_p$ is $p$-elementary. Thus (2) follows.
This proves the lemma.
\end{proof}
The following is the key to the proof of Theorem \ref{Th2.1}.
\begin{proposition}\label{Prop2.2}
Let $R$ be a Dynkin type with $\rank (R)=20$, and
let $L$ be an even overlattice of $L(R)$. Suppose that $p$ is an odd prime and that
$p \not| \ \disc(L)$.
\begin{itemize}
\item[(1)]
Suppose that $L \rightarrow \Lambda$ is a primitive embedding
into an even $p$-elementary lattice of signature $(1, 21)$
with non-cyclic $\Lambda^{\vee}/\Lambda$.
Let $T = L^{\perp}$ be the orthogonal complement
of $L$ in $\Lambda$.
Then {\rm (1a) $\sim$ (1e)} below hold.
\begin{itemize}
\item[(1a)]
$T$ is an even lattice of signature $(1, 1)$ such that $\disc(T) = -p^2 \disc(L)$ and
$T^{\vee}/T \cong \Z /(p \ell_1) \oplus \Z /(p \ell_2)$.
\item[(1b)]
There are a canonical isomorphism
$\varphi: L^{\vee}/L \rightarrow (T^{\vee}/T)_{p'}$
and the relation $(q_T)_{p'} = -q_L$ (after the
identification via $\varphi$).
\item[(1c)]
Write $T = \Z t_1 \oplus \Z  t_2$. Then $(t_i.t_j) = p(t_{ij}')$
for some $t_{ij}' \in \Z $.
\item[(1d)]
Let $T' = \Z t_1' \oplus \Z   t_2'$ be the lattice with the
intersection form $(t_i' . t_j') = (t_{ij}')$. Then
$(D_{T\sprime}, p q_{T\sprime}) \cong ((D_T)_{p'}, (q_{T})_{p'})
\cong (D_{L}, -q_{L})$.
\item[(1e)]
$\Lambda$ is the unique even $p$-elementary lattice of signature $(1, 21)$
and discriminant $-p^2$.
\end{itemize}
\item[(2)]
Conversely, suppose that $T$ is a lattice
of signature $(1, 1)$ satisfying
{\rm (1a) and (1b)}. Then there is a primitive embedding
$L \rightarrow \Lambda$ into the unique $p$-elementary even lattice $\Lambda$
of signature $(1, 21)$ and determinant $-p^2$
such that $T$  is isomorphic to the orthogonal complement $L^{\perp}$ of $L$ in $\Lambda$.
\end{itemize}
\end{proposition}
\begin{proof}
Consider the inclusions:
\begin{equation}\label{eq:second*}
 L \oplus T \subset \Lambda \subset \Lambda^{\vee}
\subset L^{\vee} \oplus T^{\vee}.
\end{equation}
Since $L \rightarrow \Lambda$ is a primitive embedding,
its dual is a surjection $\Lambda^{\vee} \rightarrow L^{\vee}$
which factors as $\Lambda^{\vee} \rightarrow L^{\vee} \oplus T^{\vee} \rightarrow L^{\vee}$
where the first is the inclusion in~\eqref{eq:second*} while the second is the projection.
This surjection induces surjections
$\Lambda^{\vee}/(L \oplus T) \rightarrow L^{\vee}/L$ and
$\varphi_1 : (\Lambda^{\vee}/(L \oplus T))_{p'} \rightarrow (L^{\vee}/L)_{p'} = L^{\vee}/L$
where the latter equality is because $\gcd(p, \disc(L)) = 1$.
Since $\Lambda$ is $p$-elementary we have $(\Lambda^{\vee}/\Lambda)_{p'} = 0$
and hence $(\Lambda^{\vee}/(L \oplus T))_{p'} = (\Lambda/(L \oplus T))_{p'}$.
Similarly, we have surjection
$\varphi_2 : (\Lambda^{\vee}/(L \oplus T))_{p'} \rightarrow (T^{\vee}/T)_{p'}$.
On the other hand, the inclusion~\eqref{eq:second*}  and the assumption that $\Lambda$ is $p$-elementary imply
$|L^{\vee}/L| \, |(T^{\vee}/T)_{p'}| = |(\Lambda/(L \oplus T))_{p'}|^2 \ge
|L^{\vee}/L| \, |(T^{\vee}/T)_{p'}|$ where the latter inequality is due to
the surjectivity of both $\varphi_i$. Thus both $\varphi_i$ are isomorphisms.
Set $\varphi = \varphi_2 \, \varphi_1^{-1}: L^{\vee}/L \rightarrow (T^{\vee}/T)_{p'}$.
For every $\overline{r'} \in L^{\vee}/L$, we write $\varphi(\overline{r'}) = \overline{t'}$,
and see that the coset of $r' + t'$
belongs to $(\Lambda/(L \oplus T))_{p'}$.
So $0 = q_{\Lambda}(\overline{r'+t'}) = q_R(\overline{r'}) + (q_{T})_{p'}(\overline{t'})$.
This proves (1b).
\par
Let $e$ be the exponent of the abelian group $L^{\vee}/L$ so that the latter is $e$-torsion.
This $e$ is coprime to $p$ by the assumption.
Then $\Lambda / (L \oplus T)$ ($\cong (L^{\vee} \oplus T^{\vee})/\Lambda^{\vee}$)
is $e$-torsion.
Indeed, for $r' + t' \in \Lambda \subset L^{\vee} \oplus T^{\vee}$,
we have, mod $L \oplus T$, that $e (r' + t') = e t' \in \Lambda \cap L^{\perp} = T$.
So $\Lambda/(L \oplus T)$ equals $(\Lambda/(L \oplus T))_{p'}$ and is isomorphic
to both $L^{\vee}/L$ and $(T^{\vee}/T)_{p'}$ via $\varphi_i$'s;
denote by $r$ the order of these three isomorphic groups,
which is coprime to $p$.
\par
We assert that $T^{\vee}/T$ is $pe$-torsion. Indeed, for
$t' \in T^{\vee}$, we have $e t' \in \Lambda^{\vee}$ and hence
$p e t' \in \Lambda \cap L^{\perp} = T$, because $\Lambda$ is $p$-elementary.
Since $T$ is of rank $2$, we have
$T^{\vee}/T \cong \Z /(p^{\varepsilon_1} \ell_1) \oplus \Z /(p^{\varepsilon_2} \ell_2)$
where $\ell_i \ge 1$,
each $\varepsilon_i \in \{0, 1\}$ and $\gcd(p, \ell_i) = 1$.
Note that $\Lambda^{\vee}/\Lambda \cong (\Z /(p))^{\oplus \lambda}$ for some
$\lambda \ge 2$.
The inclusion~\eqref{eq:second*} above implies that
$- p^{\varepsilon_1 + \varepsilon_2} \ell_1 \ell_2 \disc(L) =
\disc(T) \, \disc(L) = r^2 \disc(\Lambda) = -p^\lambda r^2$.
So $|\Lambda| = -p^2$ and both $\varepsilon_i = 1$.
Note also that $\ell_1 \ell_2 = |(T^{\vee}/T)_{p'}| = r = \disc(L)$.
This proves (1a) and (1e). The assertion (1c) is the observation
in (1a) that $(T^{\vee}/T)_p$ is $p$-elementary, $\rank(T) = 2$
and $\disc(T) = p^2 \times$ (some integer coprime to $p$)
and that the calculation of $T^{\vee}/T$ is essentially
the calculation of the matrix $(t_i . t_j)^{-1}$; see Lemma \ref{TT'}.
The assertion (1d) follows from (1b) and Lemma \ref{TT'}
by noting that $p$ does not divide $\disc(T') = -\disc(L)$.
\par
Next we prove Proposition \ref{Prop2.2} (2).
We define an overlattice $\Gamma$ of $L \oplus T$
by adding elements $r' + t' \in L^{\vee} \oplus T^{\vee}$ such that
$\varphi(\overline{r'}) = \overline{t'}$.
Note that $q_{L \oplus T}(\overline{r'+t'}) = q_L(\overline{r'}) + (q_T)_{p'}(\overline{t'}) = 0$,
so $\Gamma$ is an even overlattice of $L \oplus T$ such that the projections
induce isomorphisms: $L^{\vee}/L \cong \Gamma/(L \oplus T) \cong (T^{\vee}/T)_{p'}$.
Now $|\Gamma| = |L \oplus T|/\disc(L)^2 = \disc(T)/\disc(L) = -p^2$ by (1a).
Consider the inclusion~\eqref{eq:second*} above but with $\Lambda$ replaced by $\Gamma$,
we see that $\Gamma^{\vee}/\Gamma$ equals $(\Gamma^{\vee}/\Gamma)_p$, i.e., it is $p$-torsion
because so is $((L^{\vee} \oplus T^{\vee})/(L \oplus T))_p = (T^{\vee}/T)_p$ by (1a).
So $\Gamma$ is $p$-elementary. It is clear from the construction that
both $L \rightarrow \Gamma$ and $T \rightarrow \Gamma$ are primitive embeddings,
whence $T = L^{\perp}$ in $\Gamma$.
This proves Proposition \ref{Prop2.2}.
\end{proof}
\begin{proof}[Proof of Theorem \ref{Th2.1}]
We now prove Theorem \ref{Th2.1}, the direction $(1) \Longrightarrow (2)$.
So there is a normal $K3$ surface $Y$ defined over an algebraically closed
field $k$ with $\ch(k) = p$ such that $\Sing(Y)$ is of Dynkin
type $R$. Let $f : X \rightarrow Y$ be the minimal resolution and
$\Ex(f)$ the reduced exceptional divisor. Then $\Ex(f)$ is
also of Dynkin type $R$.
Since the Picard number $\rho(X) = \rho(Y) + \#\Ex(f) \ge 21$,
we have $\rho(X) = 22$ (see Artin \cite{Ar}) and hence $X$ is supersingular
in the sense of Shioda \cite{Shioda}. Let $\Lambda$ be the N\'eron-Severi lattice $\NS(X)$ of $X$.
Then
$\Lambda$ is $p$-elementary and $|\Lambda| = -p^{2 \sigma}$ where
$1 \le \sigma \le 10$ is the Artin invariant (Corollary~\ref{cor:RS2}).
Let $L$ denote
the sublattice of $\Lambda$ spanned by numerical equivalence classes of
irreducible components in $\Ex(f)$.
Then we have $L\cong L(R)$.
Let $\wt{L}$ be the closure of the sublattice $L$ in $\Lambda$.
Applying Proposition \ref{Prop2.2} to the primitive
embedding $\wt{L} \rightarrow \Lambda$, we see that
$\sigma = 1$. So Theorem \ref{Th2.1} (2) is true.
\par
Next we prove Theorem \ref{Th2.1}, the direction
$(2) \Longrightarrow (3)$. We use the notation above.
We assert that $\roots(\wt{L}) = \roots(L)$.
Indeed, suppose that $v \in \wt{L}$ is a $(-2)$-vector.
By considering $-v$ and the Riemann-Roch theorem, we may assume that
$v$ is represented by an effective divisor $V$ on $X_p$.
Since this $V$ is perpendicular to the pull back of an ample divisor
on $Y$, our $V$ is contractible to a point on $Y$, whence
$v = [V]$ belongs to $L$. The assertion is proved.
The rest of (3) follows from Proposition \ref{Prop2.2} applied
to the primitive embedding $\wt{L} \rightarrow \Lambda$.
\par
Finally, we prove Theorem \ref{Th2.1}, the direction
$(3) \Longrightarrow (1)$.
Define $T$ as in Lemma \ref{TT'}. Then
Propositions \ref{Prop2.2} (1a) and (1b) are satisfied
by $\wt{L}$ and $T$.
By Proposition \ref{Prop2.2} (both assertions there), there is a primitive embedding
$\wt{L} \rightarrow \Lambda$ into the unique even $p$-elementary
lattice of signature $(1, 21)$ and discriminant $-p^2$ such that $T = \wt{L}^{\perp}$
in $\Lambda$.
We have
$\NS(X_p) = \Lambda$.
Take a primitive element $v$ in $T^{\vee}$ such that $v^2 < -2$.
Let $h$ be a generator of $v^{\perp} \cap T^{\vee}$. So $h^2 > 0$.
We claim that
\begin{equation}\label{eq:rootss}
\roots(h^{\perp} \cap \Lambda) = \roots(\wt{L}).
\end{equation}
It is clear that the left-hand side of~\eqref{eq:rootss} includes the right-hand side of~\eqref{eq:rootss}.
Let $u$ be in the left-hand side of~\eqref{eq:rootss}. Write $u = r' + t'$ with $r' \in \wt{L}^{\vee}$
and $t' \in T^{\vee}$. Then $0 = h . u = h . t'$, whence
$t' \in T^{\vee} \cap h^{\perp} = \Z  [v]$.
So $t' = m v$ for some integer $m$. If $m \ne 0$, then
$-2 = u^2 = (r')^2 + (t')^2 \le m^2 v^2 < -2$, absurd.
So $m = 0$ and $u = r' \in \wt{L}^{\vee} \cap \Lambda = \wt{L}$
and hence $u \in$ the right-hand side of~\eqref{eq:rootss}. The claim is proved.
By considering $-h$ and isometry of $\Lambda$, we may assume
that a positive multiple of $h$ is represented by a nef and big Cartier
divisor $H$ on $X_p$
(see Rudakov-Shafarevich \cite{RS1}). Note that $|2H|$ is base point
free (see Nikulin \cite{Ni}, Proposition 0.1 and Saint-Donat \cite{SD},
Corollary 3.2). Let $f : X_p \rightarrow Y$ be the birational morphism
onto a normal surface, which is the Stein factorization of $\Phi_{|2H|}
: X_p \rightarrow \P^N$ with $N = \dim |2H|$.
Then $f$ is nothing but the contraction of all the curves
perpendicular to $H$. So by the genus formula and the Riemann-Roch theorem,
$\Ex(f)$ contains and consists of all curves representing elements in $\roots(h^{\perp} \cap \Lambda)
= \roots(\wt{L}) = \roots(L(R))$, whence $\Ex(f)$ is of Dynkin type $R$.
Thus $Y$ is a normal $K3$ surface with $\Sing(Y)$ of Dynkin type $R$.
Hence the assertion (1) is true. This completes the proof of Theorem \ref{Th2.1}.
\end{proof}
The following result
imposes on a Dynkin diagram $R$ of rank $20$
a necessary condition for the set of $R$-supersingular $K3$ primes to be non-empty.
The proof follows from the proof of Dolgachev-Keum \cite[Lemma 3.2]{DK}.
We reprove it here for the convenience of the readers.
\begin{lemma}\label{sqfree}
Let $R$  be a Dynkin diagram of rank $20$.
If the set of $R$-supersingular $K3$ primes is non-empty,
then $\disc(R)$ is not a perfect square.
\end{lemma}
\begin{proof}
By the assumption, there exist a $K3$ surface $X$
and $20$ smooth rational curves on $X$ whose numerical equivalence  classes span a
sublattice $L(R) \subset \NS(X)$ of Dynkin type $R$ and
which are contractible to rational double points on
a normal $K3$ surface $Y$. Since $\rho(X) \ge 1 + 20$,
we have $\rho(X) = 22$ and $X$ is supersingular.
Let $T$ denote the orthogonal complement of  $L(R)$ in $\NS(X)$.
Then $L(R) \oplus T$ is a sublattice of $\NS(X)$ of index $a$ say.
So we have:
\begin{equation}\label{eq:*}
\disc(R) \ \disc(T) = - a^2 p^{2 \sigma(X)},
\end{equation}
where $\disc(\NS(X)) = -p^{2 \sigma(X)}$ with
$\sigma(X) \in \{1, 2, \dots, 10\}$ the Artin invariant. 
\par
Suppose the contrary that $\disc(R)$ is a perfect square.
Then~\eqref{eq:*} implies that $- \disc(T)$ is a perfect square too.
By Conway-Sloane \cite{CS}, Chapter 15, Section 3, $T$ represents
zero: there is a non-zero vector $t$ in $T$ with $t^2 = 0$.
We may assume that $t$ is primitive in $T$.
By the Riemann-Roch theorem, we may assume that $t$ is represented
by an effective divisor. Applying Rudakov-Shafarevich \cite{RS2}, Chapter 3,
Proposition 3, there is a composite  $\sigma : \NS(X) \rightarrow \NS(X)$
of reflections with respect to $(-2)$-vectors
such that $\sigma(t)$ is represented by a general fibre $F$
of an elliptic or quasi-elliptic fibration $\varphi : X \rightarrow \P^1$.
There is a natural inclusion below where the lattice $F^{\perp}$ is the orthogonal in $\NS(X)$
of $\Z [F]$:
$$\sigma(L(R)) \rightarrow F^{\perp}/\Z [F].$$
Since $\rank(L(R)) = 20$, we can write
$F^{\perp}/\Z [F] \cong K_1 \oplus \dots \oplus K_r$
which includes $\sigma(L(R))$ as a sublattice of finite index $b$ say;
whence
\begin{equation}\label{eq:**}
 \disc(R) = b^2 \ \prod_{\ell=1}^r \disc(K_{\ell});
\end{equation}
moreover, each $K_{\ell}$ is of Dynkin type $A_{n(\ell)}$,
$D_{n(\ell)}$, or $E_{n(\ell)}$ so that $\varphi$
has reducible fibres of type $\widetilde{K_{\ell}}$ in the notation
of Cossec-Dolgachev \cite{CD};
see the reasoning below and the proof of Lemma 2.2 in Kondo~\cite{Ko}.
Let $j : J \rightarrow {\P}^1$ be the Jacobian fibration of $\varphi$
so that $j$ and $\varphi$ have the same type of singular fibres.
We note that $\rho(J) = 2 + 20$,
$J$ is supersingular, and $J$ has a torsion Mordell-Weil group $MW(j)$.
By Shioda \cite{Shioda2}, Theorem 1.3, we have
\begin{equation}\label{eq:***}
\prod_{\ell=1}^r \disc(K_{\ell}) = -p^{2 \sigma(J)} |MW(j)|^2,
\end{equation}
where $\disc(\NS(J)) = - p^{2 \sigma(J)}$ with $\sigma(J) \in \{1, 2, \dots, 10\}$.
Now~\eqref{eq:***} and~\eqref{eq:**} imply that $p$ divides $\disc(R)$, a contradiction.
Thus the lemma is proved.
\end{proof}
\section{Finite quadratic forms and prime integers}\label{sec:fqf}
Let $q$ and $q\sprime$ be  quadratic forms
on a finite abelian group $D$.
We denote by $d$ the order of $D$.
In this section,
we consider the set $K(q, q\sprime)$
of odd prime integers $p$ which are prime to $d$
such that $(D, pq)$ is isomorphic to $(D, q\sprime)$.
\par
%
%
For a prime $l$, we put
$$
T_l:=\begin{cases}
(\Z/8\Z)\sptimes  & \textrm{if $l=2$,}\\
\{1,-1\} & \textrm{if $l\ne 2$,}
\end{cases}
$$
and for an odd prime $p\ne l$, we define $\tau_l(p)\in T_l$ by
$$
\tau_l(p):=\begin{cases}
p\;\bmod 8 & \textrm{if $l=2$,}\\
\left(\frac{p}{l}\right) & \textrm{if $l\ne 2$.}
\end{cases}
$$
We then put
$$
T_d:=\prod_l T_l
$$
where $l$ runs through the prime divisors of $d$,
and put
$$
\tau_d(p):=(\tau_l(p))\;\;\in\;\; T_d
$$
for an odd  prime integer $p$ prime to $d$.
\begin{proposition}\label{prop:onlyfinite}
Let $p_1$ and $p_2$ be odd prime integers which are prime to $d$.
If $\tau_d(p_1)=\tau_d (p_2)$, then
saying $p_1\in K(q, q\sprime)$ is equivalent to saying $p_2\in K(q, q\sprime)$
\end{proposition}
\begin{proof}
It is enough to prove that $(D, p_1q)$ and  $(D, p_2 q)$ are isomorphic.
Let $l$ be an odd prime divisor of $d$,
and let $\nu$ be the largest  integer such that $l^\nu | d$.
It follows from $\tau_l (p_1)=\tau_l(p_2)$
that there exists an integer $a_l$
such that $p_1\equiv a_l^2 p_2\bmod\,l^\nu$ holds.
Note that $a_l$ is prime to $l$.
Suppose that $d$ is even.
It follows from $\tau_2 (p_1)=\tau_2(p_2)$
that there exists an integer $a_2$
that satisfies $p_1\equiv a_2^2 p_2\bmod\,2^{\mu+1}$,
where $\mu$ is the largest integer such that $2^\mu|d$.
Note that $a_2$ is odd.
By the Chinese Remainder Theorem,
we have an integer $a$  that satisfies
$a\equiv a_l\bmod\, l^\nu$ for each odd prime divisor $l$  of $d$,
and
$$
a\equiv\begin{cases}
a_2\bmod\, 2^{\mu+1} &\textrm{if $d$ is even,}\\
\rlap{1}\phantom{a_2}\bmod\, 2 &\textrm{if $d$ is odd.}
\end{cases}
$$
Then we have
 $p_1\equiv a^2 p_2 \bmod\, 2d$.
 Note that $a$ is prime to $d$.
 Since $b[q](x, x)=q(x) \bmod\,\Z$,
 $q(x)$ is contained in $(1/d)\Z/2\Z\subset\Q/2\Z$
 for any $x\in D$.
 Therefore we have
 $$
 p_1 q=a^2 p_2 q.
 $$
 The multiplication by $a$ induces an automorphism $\alpha $ of $D$.
 Since $\alpha^* ( p_2 q)=p_1 q$,
 we see that $p_1q$ and $p_2 q$ are isomorphic.
\end{proof}
\begin{corollary}
There exists a subset $S(q, q\sprime)$ of $T_d$ such that
$K(q, q\sprime)$ is equal to the set of odd prime integers $p$ which are prime to $d$
such that $\tau_d (p)\in S(q, q\sprime)$.
\end{corollary}
%
%
%
%
%
%
%
%
%
%
%
\section{Algorithm}\label{sec:algorithm}
Let $R$ be a Dynkin type of rank $20$.
We put
$$
T(R):=\prod_{l} T_l,
$$
where $l$ runs through the set of prime divisors of $\disc (R)$,
and for an odd prime $p$ that does not divide $\disc (R)$,
we put
$$
\tau (p):= (\tau_l (p))_l\in T(R).
$$
In this section, we present an algorithm to obtain a subset
$S(R)\subset T(R)$ with the following property:
an odd prime $p$ that does not divide $\disc (R)$ is an $R$-supersingular $K3$ prime
if and only if $\tau (p)\in S(R)$.
\subsection{Step 1}
We first calculate the  set $\OLL(R)$ of even overlattices of $L(R)$ using
Proposition 1.4.1 of  Nikulin~\cite{Nikulin}.
For each even overlattice $\wt{L}$ of $L(R)$, we can calculate the set $\Roots (\wt{L})$  of roots
of $\wt{L}$ by the method described in~\cite{ShimadaEll},~\cite{ShimadaOdd} or~\cite{MR3166075}.
Comparing $\Roots (\wt{L})$
with $\Roots (L(R))$ for each $\wt{L}$,
we make the set $\OL (R)$.
\subsection{Step 2}
We calculate the discriminant group $D_{\wt{L}}$
for each $\wt{L}\in \OL (R)$,
and make the set $\OL\sprime (R)$ of all $\wt{L}\in \OL (R)$
such that the length of $D_{\wt{L}}$
 is $\le 2$.
For each $\wt{L}\in \OL\sprime (R)$,
we calculate the isomorphism class of the
finite quadratic form $(D_{\wt{L}}, -q_{\wt{L}})$.
\subsection{Step 3}
For each $\wt{L}\in \OL\sprime  (R)$, we do the following calculation.
We put $d:=\disc (\wt{L})$,
which is a positive integer.
First we make the list $\TTT(d)$ of isomorphism classes of even indefinite lattices
 $T\sprime$ of rank $2$
with discriminant $-d$
using the classical theory of binary forms due to  Gauss.
(See~\cite[Chapter 15, \S3.3]{CS}.)
For each $T\sprime\in \TTT(d)$ we calculate the discriminant group $D_{T\sprime}$ of $T\sprime$.
If $D_{T\sprime}$ is isomorphic to $D_{\wt{L}}$,
then we calculate the set
$$
S(\wt{L}, T\sprime):=\prod_{l|\disc (R)} S_l(\wt{L}, T\sprime)\;\;\subset\;\; T(R)
$$
such that $(D_{T\sprime}, pq_{T\sprime})$ is isomorphic to $(D_{\wt{L}}, -q_{\wt{L}})$
if and only if $\tau _l (p)\in S_l(\wt{L}, T\sprime)$ for each
prime divisor $l$ of $\disc (R)$.
In virtue of Proposition~\ref{prop:onlyfinite},
we have to check only a finite number of prime integers.
(Note that
the set of prime divisors of $|D_{\wt{L}}|$ is a subset of the set of prime divisors of $\disc (R)=|D_{L(R)}|$.
If a prime divisor $l$ of $\disc (R)$
does not divide  $\disc (\wt{L})$, then we put $S_l(\wt{L}, T\sprime)=T_l$.)
If $D_{T\sprime}$ is not isomorphic to $D_{\wt{L}}$,
then we put $S(\wt{L}, T\sprime)=\emptyset$.
\par
%
%
The set $S(R)$ is the union of all $S(\wt{L}, T\sprime)$,
where $\wt{L}$ runs through the set $\OL\sprime (R)$ and $T\sprime$ runs through the set
$\TTT(\disc (\wt{L}))$.
\section{Example}\label{sec:detail}
\newcommand{\diag}{\mathord{\rm diag}}
We will explain the case
$$
R:=D_7+A_{11}+2 A_1
$$
in detail.
The discriminant form of
the negative-definite root lattice $L(R)$ is expressed by the diagonal matrix
$$
\diag[-7/4, -11/12, -1/2, -1/2]
$$
with respect to  the basis of the discriminant group
$$
D_{L(R)}\;\;\cong\;\; \Z/4\Z\;\times\; \Z/12\Z\;\times\; \Z/2\Z\;\times\;\Z/2\Z
$$
given    in~\cite[Section 6]{ShimadaEll}.
There are eight  isotropic vectors in $D_{L(R)}$:
$$
\begin{array}{l}
0:=[0,0,0,0],\;\;
v_1:=[0,6,1,1],\;\;
v_2:=[1,3,0,0],\;\;
v_3:=[1,9,0,0],\;\;
v_4:=[2,0,1,1],\\
v_5:=[2,6,0,0]=2v_2=2v_3,\;\;
v_6:=[3,3,0,0]=-v_3,\;\;
v_7:=[3,9,0,0]=-v_2.
\end{array}
$$
Let $L_{(i)}$ be the even overlattice of $L(R)$ corresponding to
the totally isotropic subgroup
of $D_{L(R)}$ generated by $v_i$.
The Dynkin type of $\Roots (L_{(i)})$
is equal to $R$ if $i\ne 4$,
while it is $A_{11}+D_9$ if $i=4$.
Hence the even overlattice $L(H)$ of $L(R)$
corresponding to an totally isotropic subgroup
$H$ satisfies $\Roots (L(H))=\Roots(L(R))$ if and only if
$v_4\notin H$.
The totally isotropic subgroups that do not contain $v_4$ are listed as follows:
$$
H_0=\{0\},\;\;
H_1=\{0, v_1\},\;\;
H_2=\{0, v_5\}, \;\;
H_3=\{0, v_2, v_5, v_7\}, \;\;
H_4=\{0, v_3, v_5, v_6\}.
$$
Let $\gamma\in \Aut(L(R))$ be the isometry of $L(R)=L(D_7)\oplus L(A_{11}+2A_{1})$
that is the multiplication by $-1$ on the factor $L(D_7)$
and the identity on $L(A_{11}+2A_{1})$.
Then the action of $\gamma$ on $D_{L(R)}$ interchanges $H_3$ and $H_4$.
Therefore the even lattices $L(H_3)$ and $L(H_4)$ are isomorphic.
The lengths of $D_{L(H_0)}$ and $D_{L(H_2)}$ are $\ge 3$,
while the lengths of $D_{L(H_1)}$ and $D_{L(H_3)}\cong D_{L(H_4)}$ are $2$.
\par
%
%
The discriminant  form of $L(H_1)$
multiplied by $(-1)$ is given by
$$
\left(\Z/4\Z\times \Z/4\Z,
\left[\begin{array}{cc}7/4 & 0 \\ 0 & 3/4 \end{array}\right]\right)
\times
\left(\Z/3\Z,
\left[\begin{array}{c}
2/3
\end{array}
\right]
\right).
$$
There are four isomorphism classes of even indefinite lattices
of rank $2$ with discriminant $-48$:
$$
S_\pm:=\left[
\begin{array}{cc}
\pm 2 & 6 \\
6 & \mp 6
\end{array}
\right],
\qquad
T_\pm:=\left[
\begin{array}{cc}
\pm 4 & 4 \\
4 & \mp 8
\end{array}
\right].
$$
The discriminant group of $S_\pm$ is isomorphic to $\Z/2\Z\times \Z/8\Z\times\Z/3\Z$,
and hence is not isomorphic to $D_{L(H_1)}$.
The discriminant forms  of $T_\pm$ are
$$
\left(\Z/4\Z\times \Z/4\Z,
\left[\begin{array}{cc}\pm 1/4 & 0 \\ 0 & \pm 5/4 \end{array}\right]\right)
\times
\left(\Z/3\Z,
\left[\begin{array}{c}
\pm 2/3
\end{array}
\right]
\right).
$$
Hence we have the following equivalence for prime integers $p\ne 2, 3$:
$$
\begin{array}{rcl}
p\in K(-q_{L(H_1)}, q_{T_+})
&\Longleftrightarrow
&p\bmod 8\equiv  3\;\textrm{or}\;7\quand(\frac{p}{3})=1,\\
p\in K(-q_{L(H_1)}, q_{T_-})
&\Longleftrightarrow
&p\bmod 8\equiv  1\;\textrm{or}\;5\quand(\frac{p}{3})=-1.
\end{array}
$$
There are two isomorphism classes of even indefinite lattices
of rank $2$ with discriminant $-12$:
$$
U_\pm:=\left[
\begin{array}{cc}
\pm 2 & 2 \\
2 & \mp 4
\end{array}
\right].
$$
By the same calculation as above,
 we have the following equivalence for prime integers $p\ne 2, 3$:
$$
\begin{array}{rcl}
p\in K(-q_{L(H_3)}, q_{U_+})
&\Longleftrightarrow
&p\bmod 8\equiv  1\;\textrm{or}\;5\quand(\frac{p}{3})=-1,\\
p\in K(-q_{L(H_3)}, q_{U_-})
&\Longleftrightarrow
&p\bmod 8\equiv  3\;\textrm{or}\;7\quand(\frac{p}{3})=1.
\end{array}
$$
Thanks to the equalities
$$
K(-q_{L(H_1)}, q_{T_+})= K(-q_{L(H_3)}, q_{U_-}),
\quad
K(-q_{L(H_1)}, q_{T_-})=K(-q_{L(H_3)}, q_{U_+}),
$$
the natural density of $R$-supersingular $K3$ primes is $1/2$.
\end{document}